\documentclass[12pt]{article}
\input epsf.tex

\usepackage{graphicx}
\usepackage{amsmath,amsthm,amsfonts,amscd,amssymb,comment,eucal,latexsym,mathrsfs}
\usepackage{stmaryrd}
\usepackage[all]{xy}

\usepackage{epsfig}
\usepackage{color}
\usepackage[all]{xy}
\xyoption{poly}
\usepackage{fancyhdr}
\usepackage{wrapfig}
\usepackage{epsfig}
 \usepackage{mathtools}

\theoremstyle{plain}
\newtheorem{thm}{Theorem}[section]
\newtheorem{prop}[thm]{Proposition}
\newtheorem{lem}[thm]{Lemma}

\theoremstyle{definition}
\newtheorem{defn}{Definition}
\theoremstyle{remark}

\advance \topmargin by -\headheight
\advance \topmargin by -\headsep
\evensidemargin \oddsidemargin
  \def\C{{\mathbb{C}}}  \def\E{{\mathbb{E}}}   \def\H{{\mathbb{H}}}          \def\R{{\mathbb{R}}}        

     \def\barf{{\bar{f}}}                    

        \def\cI{{\mathcal{I}}}                 

\renewcommand{\hat}{\widehat}

  \def\hC{{\widehat{C}}}             \def\hP{{\widehat{P}}}   \def\hS{{\widehat{S}}}  \def\hU{{\widehat{U}}} \def\hV{{\widehat{V}}}    
     \def\hf{{\hat{f}}}                    

\renewcommand{\tilde}{\widetilde}

     \def\tf{{\tilde{f}}}          \def\tp{{\tilde{p}}}        \def\tx{{\tilde{x}}} \def\ty{{\tilde{y}}} 

 \def\tgamma{{\widetilde{\gamma}}}               

\newcommand{\G}{\Gamma}

\newcommand{\La}{\Lambda}

\newcommand{\Si}{\Sigma}
\newcommand{\eps}{\epsilon}

\renewcommand\a{\alpha}
\renewcommand\b{\beta}
\renewcommand\d{\delta}
\newcommand\g{\gamma}

\renewcommand\l{\lambda}

\newcommand\area{{\operatorname{area}}}

\newcommand\Isom{\operatorname{Isom}}

\def\cc{{\curvearrowright}}

  \newcommand{\actson}{\curvearrowright}

\newcommand{\deee}{\hspace{2 pt} \mathrm{d}}
  
\begin{document}
\title{Failure of the $L^1$ pointwise ergodic theorem for $\mathrm{PSL}_2(\mathbb{R})$}
\author{Lewis Bowen\footnote{supported in part by NSF grant DMS-1500389} \,and Peter Burton\footnote{supported by an R.H. Bing Fellowship} \\ University of Texas at Austin}
\maketitle

\begin{abstract}
Amos Nevo established the pointwise ergodic theorem in $L^p$ for measure-preserving actions of $\mathrm{PSL}_2(\mathbb{R})$ on probability spaces with respect to ball averages and every $p>1$. This paper shows by explicit example that Nevo's Theorem cannot be extended to $p=1$.  
\end{abstract}


\noindent
\tableofcontents

\section{Introduction}
Birkhoff's ergodic theorem is that if $T:(X,\mu) \to (X,\mu)$ is a measure-preserving transformation of a standard probability space and $f \in L^1(X,\mu)$ then for a.e. $x \in X$, the time-averages $(n+1)^{-1}\sum_{i=0}^n f(T^i x)$ converge to the space average $\E[f|\cI(T)](x)$ (this is the conditional expectation of $f$ on the sigma-algebra of $T$-invariant measurable subsets). In particular, if $T$ is ergodic then $(n+1)^{-1}\sum_{i=0}^n f(T^i x) \to \int f \mathrm{d}\mu$ for a.e. $x$. 

To generalize this result, one can replace the single transformation $T$ with a group $G$ of transformations and the intervals $\{0,\ldots, n\}$ with a sequence of subsets of $G$ or more generally, with a sequence of probability measures on $G$. To be precise, a sequence $\{\eta_n\}_{n=1}^\infty$ of probability measures on an abstract group $G$ is {\bf pointwise ergodic in $L^p$} if for every measure-preserving action $G \cc (X,\mu)$ on a standard probability space and for a.e. $x \in X$, the time-averages
$$\int f(gx) \deee \eta_n(g)$$
converge to the space average $\E[f|\cI(G)](x)$ as $n\to\infty$ where $\E[f|\cI(G)]$ is the conditional expectation of $f$ on the sigma-algebra of $G$-invariant measurable subsets. If the measure $\eta_n$ is uniformly distributed over a ball then the time-averages are called ball-averages. 

Pointwise ergodic theorems for amenable groups with respect to averaging over F\o lner sets were established in a variety of special cases culminating in Lindenstrauss' general theorem \cite{lindenstrauss-2001}. This theorem also holds for $L^1$-functions.  Nevo and co-authors established the first pointwise ergodic theorems for free groups \cite{MR1266737, nevo-stein-birkhoff} and simple Lie groups \cite{MR1301188, MR1430433, MR1454704, margulis-nevo-stein} with respect to ball and sphere averages. See also \cite{MR2186253, gorodnik-nevo-book} for surveys. These results hold in $L^p$ for every $p>1$. It was open problem whether ball-averages could be pointwise ergodic in $L^1$ for any non-amenable group. 

Terrence Tao showed by explicit example that the pointwise ergodic theorem fails in $L^1$ for actions of free groups with respect to ball averages \cite{MR3482275}. His technique was inspired by Ornstein's counterexample demonstrating the failure of the maximal ergodic theorem in $L^1$ for iterates $P^n$ of a certain well-chosen self-adjoint Markov operator \cite{MR0236354}. 

This note proves the analogous theorem for $\mathrm{PSL}_2(\mathbb{R})$ in place of free groups. Our approach is based on the geometry of hyperbolic surfaces. In the abstract, there is a lot in common with Tao's approach but the details of the construction are significantly different. It seems likely that our methods will generalize beyond $\mathrm{PSL}_2(\mathbb{R})$.

\subsection{The main theorem}

To make the result precise, we need to introduce some notation. The {\bf hyperbolic plane} $\H^2$ is a complete, simply-connected Riemannian surface with constant curvature $-1$. It is unique up to isometry. Its orientation-preserving isometry group is isomorphic to $G:=\mathrm{PSL}_2(\mathbb{R})$. Fix a base-point $p_0 \in \H^2$. Let $F_r \subset G$ be the set of all $g$ such that $d_{\H^2}(p_0,gp_0) \le r$. 

 Given a probability-measure-preserving (pmp) action $G \actson (X,\mu)$, $r > 0$, a function $f \in L^1(X,\mu)$ and $x \in X$ the \textbf{ergodic average} is defined by \[ (\mathsf{A}_r f)(x) = \lambda(F_r)^{-1} \int_{F_r} f(g \cdot x) \deee \lambda(g) \]
where $\l$ is the Haar measure on $G$.  The \textbf{terminal maximal average} is defined by $(\mathsf{M} f)(x) = \sup_{r \geq 1} (\mathsf{A}_r |f|)(x)$. Nevo proved \cite{MR1301188}:

\begin{thm}[Nevo] \label{thm.2} Let $G \actson (X,\mu)$ be an ergodic pmp action, $p >1$ and $f \in L^p(X,\mu)$. Then  \[ \lim_{r \to \infty} (\mathsf{A}_rf)(x) = \int_X f(x) \deee \mu(x) \] for $\mu$-almost every $x \in X$. \end{thm}

The main theorem of this paper is that Nevo's Theorem does not extend to $p=1$:

\begin{thm} There exists an ergodic pmp action $G \actson (X,\mu)$ and a nonnegative function $f \in L^1(X,\mu)$ such that $(\mathsf{M} f)(x)$ is infinite for almost every $x \in X$. In particular, for almost every $x \in X$ the averages $(\mathsf{A}_rf)(x)$ fail to converge as $r \to \infty$.  \label{thm.1}  \end{thm}

\subsection{A rough overview of the construction}

Ornstein's counterexample in \cite{MR0236354} shows that the maximal ergodic theorem fails in $L^1$ for powers of a certain self-adjoint operator $P^n$. The example consists of an $L^1$-function $f$ with many components $f_i$, each of which comes with a ``time delay'' which means that $P^nf_i$ is roughly singular unless $n$ is very large (depending on $i$). This allows the amplitude of $f_i$ to be slightly smaller than would otherwise be necessary to make $\sup_n P^nf$ large on a set of significant measure. 

The example here is similar in spirit although the implementation is based on the geometry of hyperbolic surfaces. The measure space is the tangent space of a hyperbolic surface. Each component function $f_i$ is constant on a neighborhood of a cusp and the time delays are instituted by gluing surfaces together with narrow ``bottlenecks''. 

Here is more detail. For every $\eps>0$, a  hyperbolic surface $S=\H^2/\G$ (for some lattice $\G<G$)  and a non-negative $f \in L^\infty(S)$ are constructed to satisfy: (1) the $L^1$-norm of $f$ is bounded by $\eps$ and (2) there is a subset $V\subset S$ with $\area(V)/\area(S)$ bounded from below such that for all $x\in V$, there is some radius $r$ so that the $r$-ball average of $f$ centered at $x$ is $\ge 1$. This latter property means: if $\tx \in \H^2$ is a point in the inverse image of $x$ under the universal cover $\pi:\H^2 \to S$ and $\tf = f \circ \pi$ is the lift of $\pi$ then the average of $\tf$ over the ball of radius $r$ centered at $x$ is at least 1. A small additional argument (which also appears in Tao's paper) finishes the proof.

These pairs $(S,f)$ are constructed inductively. Given a pair $(S,f)$ for some $\eps>0$ (with some additional structure), a new pair $(\hS,\hf)$ is constructed satisfying roughly the same maximal function lower bounds as $(S,f)$ so that $\|\hf\|_1 \le \|f\|_1(1-\|f\|_1/6)$ (up to a small multiplicative error). By iterating this construction, the $L^1$-norm of the function can be made arbitrarily close to zero. 

The new pair $(\hS,\hf)$ is constructed from $(S,f)$ as follows. We take two isometric copies of $(S,f)$, deform them by stretching cusps into geodesics and then glue them to a pair of pants with a cusp to obtain $\hS$. The new surface has two large subsurfaces $S^{(1)}, S^{(2)}$ (each of which is isometric to a large subsurface of $S$) connected by a long narrow ``neck'' which is actually a pair of pants with a cusp. There are also two copies of $f$, denoted $f^{(1)}$ and $f^{(2)}$ supported on $S^{(1)}, S^{(2)}$ respectively. By choosing the neck to be very narrow, a continuity argument shows that the ball averages of each $f^{(i)}$ in $\hS$ are close to the ball  averages of $f$ in $S$. Theorem  \ref{thm.2} shows that if $t>0$ is chosen sufficiently large then for most $p$ in $S^{(2)}$, the radius $(r+t)$-ball averages of $f^{(1)}$ around $p$ are close to its space average $\int f^{(1)}~d\nu_{\hS}$ (for every $r>0$). 

Finally, we replace $f^{(2)}$ by ``flowing'' it for time $t$ into the cusps of $S^{(2)}$ and scaling it by a factor of $e^t[1- \int f^{(1)}~d\nu_{\hS}]$. Let $f'$ be the new function. The radius-$(r+t)$ ball averages of $f'$ are, up to small errors, equal to the radius-$r$ ball averages of $f^{(2)}$ multiplied by  $[1- \int f^{(1)}~d\nu_{\hS}]$. So let $\hf=f^{(1)}+f'$. Then we have controlled the maximal ball averages of $\hf$ on both $S^{(1)}$ and $S^{(2)}$ and the norm of $\hf$ is bounded by $\|f\|_1(1-\|f\|_1/6)$, finishing the argument.

\section{Quantitative counterexample}\label{step1}

This section reduces Theorem \ref{thm.2} to the next lemma (which is similar to \cite[Theorem 2.1]{MR3482275}). 

\begin{lem}\label{lem}  There exists a constant $b > 0$ with the following property. For every $\epsilon > 0$ there exists a weakly mixing pmp action $G \actson (Y,\eta)$ and a nonnegative function $f \in L^\infty(Y,\eta)$ such that $\|f\|_1 \leq \epsilon$ and $\eta(\{y \in Y: (\mathsf{M} f)(y) \geq 1 \}) \geq b$. \end{lem}

\begin{proof}[Proof of Theorem \ref{thm.1} from Lemma \ref{lem}]
By Lemma \ref{lem} for each $k \in \mathbb{N}$ there exist a weakly mixing pmp action $G \actson (Y_k,\eta_k)$ and a nonnegative function $f'_k \in L^\infty(Y_k,\eta_k)$ such that $\|f'_k\|_1 \leq \left(\frac{1}{2^k}\right)^2$ and if $E_k =  \{y \in Y_k: (\mathsf{M} f'_k)(y) \geq 1\} $ then $\eta_k(E_k) \geq b$. 

Let $f_k=2^k f'_k$. So $\|f_k\|_1 \leq \frac{1}{2^k}$ and $E_k =  \{y \in Y_k: (\mathsf{M} f_k)(y) \geq 2^k\}$. Let $(X,\mu)$ be the product measure space $(X,\mu):=\prod_{k =1}^\infty (Y_k,\eta_k)$. Because each action $G \cc (Y_k,\eta_k)$ is weakly mixing, the diagonal action $G \cc (X,\mu)$ is ergodic.  Let $p_k:X\to Y_k$ be the projection onto the $k^{\mathrm{th}}$ coordinate and define $\hat{f}_k = f_k \circ p_k \in L^\infty(X,\mu)$. Let $\hat{f} = \sum_{k=1}^\infty  \hat{f}_k$. Then  $\|\hat{f}_k\|_1 = \|f_k\|_1 \leq \frac{1}{2^k}$ so that $\|\hat{f}\|_1 \leq \sum_{n=1}^\infty \frac{1}{2^k} =1$.

Let $\hat{E}_k = p_k^{-1}(E_k) \subseteq X$ and, for a point $x \in X$, let $N(x) = \bigl\{k \in \mathbb{N}: x \in \hat{E}_k \bigr\}$.  Since the events $(\hat{E}_k)_{k =1}^\infty$ are independent and $\sum_{k=1}^\infty \mu(\hat{E}_k) = \sum_{k=1}^\infty \eta_k(E_k) = \infty$, the converse Borel-Cantelli Lemma implies that $N(x)$ is infinite for almost every $x \in X$. 

Since each $\hat{f}_k$ is non-negative,
$$(\mathsf{M} \hat{f})(x) \geq \sup_{k\ge 1} (\mathsf{M} \hat{f}_k)(x).$$
Therefore $(\mathsf{M} \hat{f} )(x) \geq 2^k$ for every $k$ such that $x \in \hat{E}_k$. Since almost every $x$ is contained in infinitely many $\hat{E}_k$, it follows that $(\mathsf{M} \hat{f})(x)=\infty$ for a.e. $x$. 

\end{proof}

\section{Geometric preliminaries}
This section reviews some standard facts needed for the next section which reduces Lemma \ref{lem} to a geometric problem. It will be convenient to identify the hyperbolic plane with the upper-half plane
$$\H^2:=\{x+iy \in \C:~ y>0\}$$
equipped with the Riemannian metric $ds^2 = \frac{dx^2+dy^2}{y^2}$. The group $\mathrm{SL}_2(\mathbb{R})$ acts on $\H^2$ by fractional linear transformations:
\begin{displaymath}
\left( \begin{array}{cc}
a & b \\
c & d \end{array}\right) z = \frac{az + b}{cz+d}.
\end{displaymath}
The kernel of this action is the subgroup $\{ \pm I\} \le \mathrm{SL}_2(\mathbb{R})$. Therefore, the quotient $\mathrm{PSL}_2(\mathbb{R}) = \mathrm{SL}_2(\mathbb{R})/\{\pm I\}$ acts on $\H^2$ as above. By abuse of notation, we will write elements of $\mathrm{PSL}_2(\mathbb{R})$ as matrices with the implicit understanding that the matrices are taken modulo $\{\pm I\}$. 

The action $\mathrm{PSL}_2(\mathbb{R}) \cc \H^2$ is transitive and the stabilizer of $i \in \H^2$ is the subgroup of rotations
\begin{displaymath}
K = \left\{  \left( \begin{array}{cc} 
\cos(\theta) & -\sin(\theta) \\
\sin(\theta) & \cos(\theta) \end{array}\right) :~ \theta \in \R \right\}.
\end{displaymath}
Therefore $\H^2$ can be identified with the quotient space $\mathrm{PSL}_2(\mathbb{R})/K$ via the map $g\cdot i \mapsto g K$. 


The action $\mathrm{PSL}_2(\mathbb{R}) \cc \H^2$ preserves the Riemannian metric. By taking derivatives, there is an induced action of $\mathrm{PSL}_2(\mathbb{R})$ on the unit tangent bundle, denoted by $T^1(\H^2)$. This action is simply-transitive. Therefore $\mathrm{PSL}_2(\mathbb{R})$ is the group of all orientation-preserving isometries of $\H^2$. 

By choosing a unit vector $v_0$ in the tangent space of $i \in \H^2$, we may identify $\mathrm{PSL}_2(\mathbb{R})$ with $T^1(\H^2)$ via the map $g \mapsto gv_0$. Thus we have a commutative diagram:
\begin{displaymath}\begin{array}{ccc}
\mathrm{PSL}_2(\mathbb{R}) & \leftrightarrow & T^1(\H^2) \\
\downarrow & & \downarrow \\
\mathrm{PSL}_2(\mathbb{R})/K & \leftrightarrow & \H^2
\end{array}
\end{displaymath}
Moreover $\mathrm{PSL}_2(\mathbb{R})$ acts by left translations on all four spaces and these actions commute with the maps.

Suppose $\G \le \mathrm{PSL}_2(\mathbb{R})$ is a discrete torsion-free subgroup. Then the quotient $ \G \backslash \H^2 \cong \G \backslash \mathrm{PSL}_2(\mathbb{R}) / K$ is a hyperbolic surface. More generally, for the purposes of this paper, a {\bf hyperbolic surface} is any Riemannian manifold isometric to a subset $S$ of a quotient $\G \backslash\H^2$  for some discrete torsion-free subgroup $\G \le \mathrm{PSL}_2(\mathbb{R})$ such that $S$ is equal to the closure of its interior.

By quotienting out the left-action of $\G$ on the four spaces above, we arrive at the following commutative diagram:
\begin{displaymath}\begin{array}{ccc}
\G \backslash \mathrm{PSL}_2(\mathbb{R}) & \leftrightarrow & \G \backslash T^1(\H^2) \\
\downarrow & & \downarrow \\
\G \backslash \mathrm{PSL}_2(\mathbb{R})/K & \leftrightarrow & \G \backslash \H^2
\end{array}
\end{displaymath}
The derivative of the covering map $\H^2 \to \G \backslash \H^2$ is $\G$-invariant. Therefore the unit tangent bundle of the surface $\G\backslash\H^2$ is canonically isomorphic with the quotient space $\G \backslash T^1(\H^2)$. Thus we have obtained an identification of $\G \backslash \mathrm{PSL}_2(\mathbb{R})$ with $T^1(\G\backslash\H^2)$.

\section{Reduction to geometry}\label{sec:reduction}

This section reduces the ergodic theory problem of Lemma \ref{lem} to a geometric problem. Towards that goal, suppose that $S=\G \backslash \H^2$ is a hyperbolic surface where $\G\le \mathrm{PSL}_2(\mathbb{R})$ is a discrete torsion-free subgroup. Let $\pi:\H^2 \to S$ denote the quotient map. For $f \in L^\infty(S)$ let $\tf = f \circ \pi$ be its lift to $\H^2$. Define the \textbf{geometric average} $\b_r(f) \in L^\infty(S)$ by
$$(\b_rf)(x) := \textrm{area}(B_r(\tx))^{-1} \int_{B_r(\tx)} \tf(y) \deee y$$
where $\tx \in X$ is any lift of $x$ (so $\pi(\tx)=x$) and $B_r(\tx)$ denotes the ball of radius $r$ centered at $\tx$. This does not depend on the choice of lift because $\pi$ is invariant under the deck-transformation group $\G$.

\begin{figure}
\begin{center} \includegraphics[width=5 in]{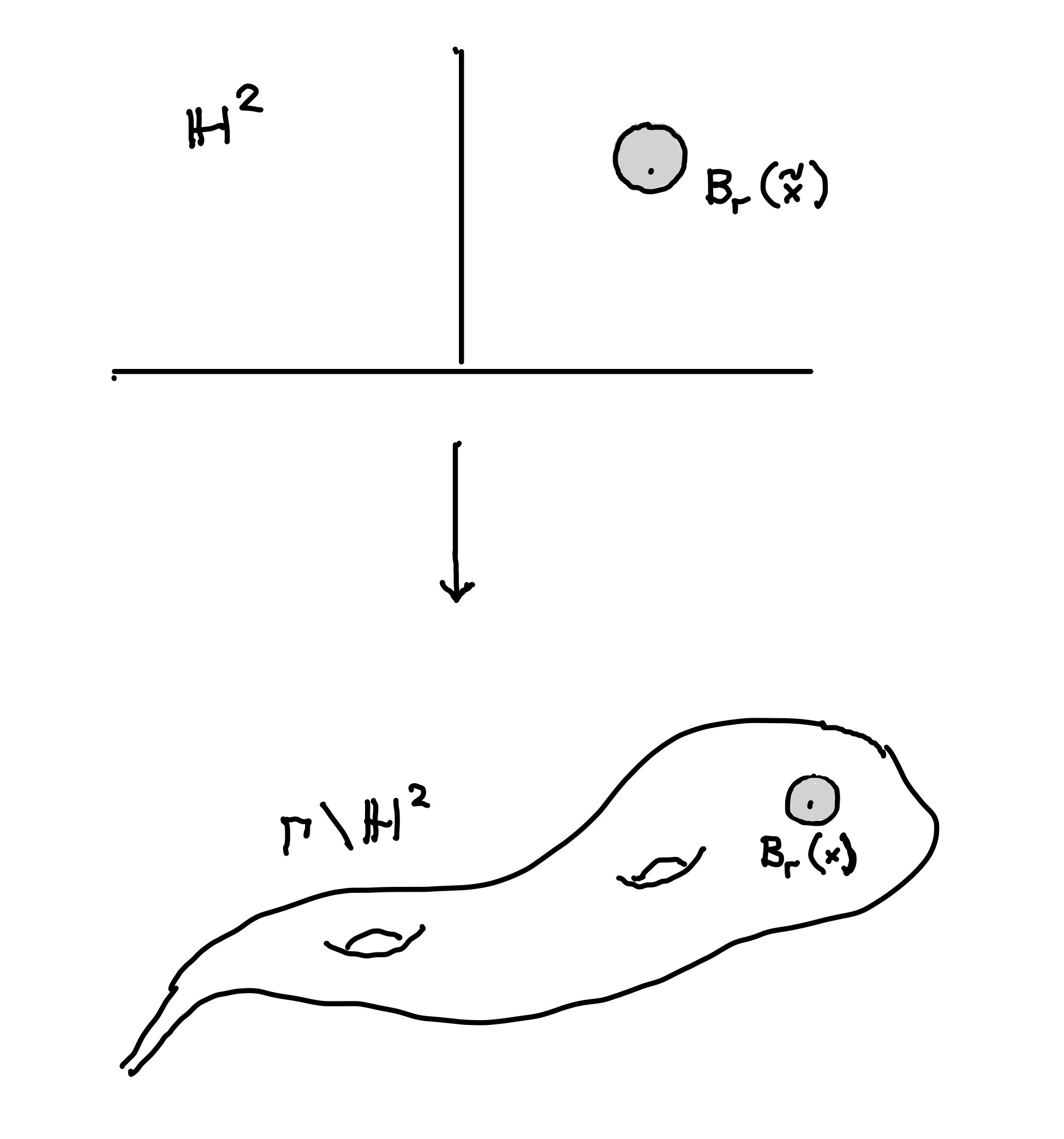} \end{center}
\caption{Geodesic balls in the hyperbolic plane and in a finite area surface}
\end{figure}

In the special case in which $S$ has finite area, let $\nu_S$ denote the hyperbolic area form on $S$ normalized so that $\nu_S(S)=1$. Also let $\|f\|_1$ denote the $L^1(S,\nu_S)$ norm. 

\begin{lem}\label{geometric}
There exists a constant $b>0$ such that for every $\eps>0$ there exists a complete connected finite-area hyperbolic surface  $S$ with empty boundary and a function $f \in L^\infty(S,\nu_S)$ satisfying
\begin{enumerate}
\item $f\ge 0$,
\item $\|f\|_1 \leq \eps$,
\item $\nu_S(\{x \in S: \sup_{r \geq 1} (\beta_rf)(x) \ge 1\}) \geq b.$
\end{enumerate}
\end{lem}


\begin{proof}[Proof of Lemma \ref{lem} from Lemma \ref{geometric}] The constant $b$ is the same in both Lemmas \ref{lem} and \ref{geometric}. Let $\epsilon > 0$ be given and let $S$ and $f$ be as in Lemma \ref{geometric}. Then $S=\G \backslash \H^2=\G\backslash \mathrm{PSL}_2(\mathbb{R})/K$ where $\G\le \mathrm{PSL}_2(\mathbb{R})$ is a torsion-free lattice. Let $\eta_S$ be the probability measure on $\G \backslash \mathrm{PSL}_2(\mathbb{R})$ given by integrating normalized Lebesgue measure on the unit circle $K$ over $\nu_S$. The right action $\mathrm{PSL}_2(\mathbb{R})$ on $\G \backslash \mathrm{PSL}_2(\mathbb{R})$ preserves $\eta_S$. We take $(Y,\eta) = (\G \backslash \mathrm{PSL}_2(\mathbb{R}),\eta_S)$. This action is ergodic because there is only orbit. It is weakly mixing because every ergodic action of $\mathrm{PSL}_2(\mathbb{R})$ is weakly mixing by the Howe-Moore Theorem \cite{MR1781937}.

If we write $q:\G \backslash \mathrm{PSL}_2(\mathbb{R}) \to S = \G \backslash \mathrm{PSL}_2(\mathbb{R})/K$ for the natural projection then $f \circ q$ is an element of $L^\infty(\G \backslash \mathrm{PSL}_2(\mathbb{R}),\eta_S)$ and $\|f \circ q\|_1 = \|f\|_1$. Let $x \in S$ and let $\xi \in q^{-1}(x)$. Then
$$(\mathsf{A}_r(f \circ q))(\xi) = (\beta_r f)(x).$$
So the action $G \cc (Y,\eta)$ and function $f \circ q$ satisfy the conclusions of Lemma \ref{lem}.
\end{proof}

\section{Pants and cusps}\label{sec:pants}

This section introduces notation to describe pants and cusps that will be useful in the main construction.  

A {\bf right-angled hexagon} is a hexagon $H$ in the hyperbolic plane such that all of its edges are geodesic segments and its interior angles are right angles. It will be convenient to label the sides of a hexagon by $f_0, e_{01}, f_1, e_{12}, f_2, e_{20}$ so that $e_{ij}$ is adjacent to both $f_i$ and $f_j$. See figure \ref{fig:hex}. 

\begin{figure}\label{fig:hex}
\begin{center} \includegraphics[width=2.5 in]{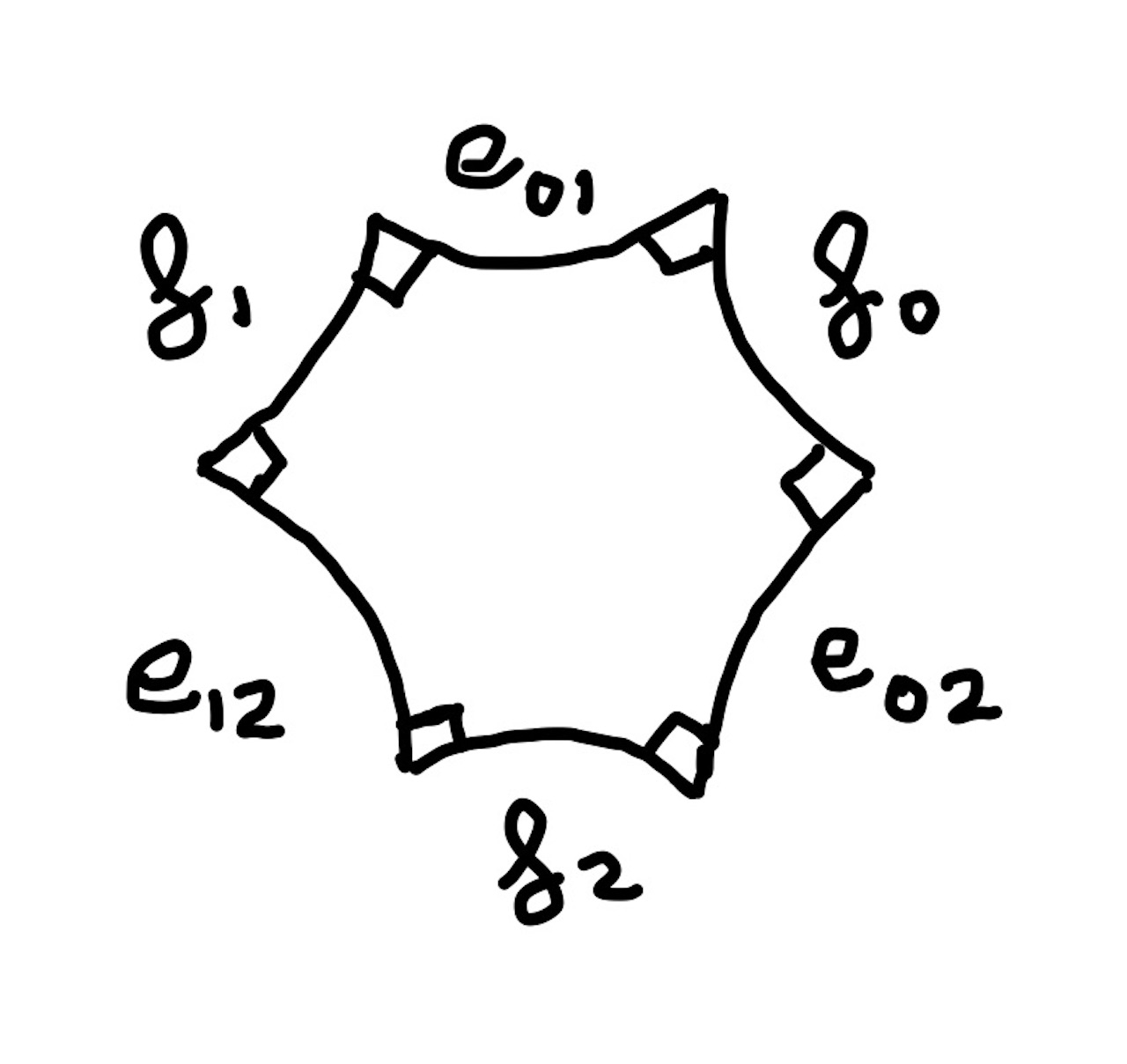} \end{center}
\caption{A right-angled hexagon}
\end{figure}

 By \cite[Theorem 2.4.2]{MR1183224}, for every triple $(l_0,l_1,l_2) \in (0,\infty)^3$ there is a right-angled hexagon $H=H(l_0,l_1,l_2)$ such that the length of $f_i$ is $l_i$ for $i \in \{0,1,2\}$. Moreover, the lengths of the other edges $(e_{ij})$ are determined by the lengths of $f_0,f_1,f_2$ so that $H$ is uniquely determined up to isometry. For example, by \cite[Theorem 2.4.1]{MR1183224},
 \begin{eqnarray}\label{hexagon}
\cosh(l_0)&=&\sinh( l_1) \sinh( l_2) \cosh( \textrm{length}(e_{12})) - \cosh( l_1) \cosh(l_2).
\end{eqnarray}
  By taking limits, we can allow $(l_0,l_1,l_2)$ to be in $[0,\infty]^3$ \cite[\S 4.4]{MR1183224}. For example, if $(l_0,l_1,l_2)=(0,0,0)$ then $H$ is an ideal triangle with its `vertices' on the boundary at infinity. We will still refer to $H$ as a right-angled hexagon even if some of its sides have zero or infinite length. 

A {\bf pair of pants} is a hyperbolic surface that is homeomorphic to a sphere minus three disjoint open disks such that each boundary component is a closed geodesic. For example, suppose for $k \in \{1,2\}$, $H^k$ is a right-angled hexagons with edges $e_{ij}^k, f_i^k$ for $i,j \in \{0,1,2\}$. In addition suppose that the length of $e_{ij}^1$ equals the length of $e^2_{ij}$ for all $i,j$ so that the hexagons are isometric. Let $P$ be the surface obtained by glueing $e_{ij}^1$ to $e_{ij}^2$ isometrically for $i,j \in \{0,1,2\}$. This is a pair of pants (for details see \cite[\S 3.1]{MR1183224} where it is called a $Y$-piece). The lengths of the boundary components are twice the lengths of the sides $f_i^k$. Conversely, if $P$ is any pair of pants with boundary components $\partial_i P$ for $i\in \{0,1,2\}$ then for every pair $\{i,j\} \in \{0,1,2\}$ there exists a unique shortest geodesic segment $\g_{ij}$ from $\partial_i P$ to $\partial_j P$. By cutting along these geodesic segments, we obtain two isometric right-angled hexagons (the {\bf canonical right-angled hexagons of $P$}). Thus for every triple of numbers $(l_0,l_1,l_2)\in (0,\infty)$ there exists a pair of pants $P$ with boundary lengths equal to $l_0,l_1,l_2$ and $P$ is unique up to isometry.  See \cite[Theorem 3.1.7]{MR1183224} for a formal proof of this statement.

A {\bf pair of pants with $k$-cusps} (for $k \in \{0,1,2,3\}$) is a hyperbolic surface that is homeomorphic to a sphere minus $k$ points and $3-k$ disjoint open disks such that each boundary component is a closed geodesic. They can be constructed exactly as in the previous paragraph by allowing the lengths of the edges $f_i^k$ to take values in $[0,\infty)$. See \cite[Lemma 4.4.1]{MR1183224} for a formal proof. 

The {\bf canonical horoball} is the subset
$$H_0:=\{x+iy \in \C:~ y\ge 1\} \subset \H^2.$$

For any $x_0 \in \R$, the map $z \mapsto z+x_0$ is an orientation-preserving isometry of the hyperbolic plane and therefore is represented as an element of $\mathrm{PSL}_2(\mathbb{R})$. A {\bf cusp} is a surface isometric to a quotient of the form $C:=H_0/\{z \mapsto z+x_0\}$ for some $x_0>0$. For example, if $P$ is a pair of pants with $k$ cusps as defined above, then there really are $k$ disjoint cusps on $P$  \cite[Proposition 4.4.4]{MR1183224}.

By Gauss-Bonet, the area of a right-angled hexagon is $\pi$. So the area of a pair of pants is $2\pi$ \cite[p.153]{MR1393195}.

\section{Deformations of surfaces}\label{sec:deform}

The proof of Lemma \ref{geometric} constructs surfaces and $L^1$-functions inductively by cutting, pasting and deforming. The main result of this section is that the averages $\b_r f$ vary continuously under deforming the boundary of surfaces equipped with additional structure. To make this precise, we need the following ad hoc definition. 

A {\bf panted surface} is a pair $(S,P)$ such that $S$ is a connected oriented hyperbolic surface and  $P \subset S$ is a closed subsurface satisfying:
\begin{itemize}
\item $P$ is a pair of pants with $\le 1$ cusp, 
\item the complement $S \setminus P$ has two connected components,
\item two of the boundary components of $P$ are contained in the interior of $S$. These are denoted by $\partial^1 P, \partial^2 P$. If there is a third boundary component then it is denoted by $\partial^0 P$. 
\end{itemize}

For $\a>0$, the {\bf $\a$-deformation} of $(S,P)$ is a panted surface $(S_\a,P_\a)$ defined as follows. Let $P_\a$ be the (compact) oriented hyperbolic pair of pants with geodesic boundary $\partial P_\a = \cup_{i=0}^2 \partial^i P_\a$ satisfying
\begin{eqnarray*}
\textrm{length}(\partial^0 P_\a) &=& \a  \\
\textrm{length}(\partial^1 P_\a) &=& \textrm{length}(\partial^1 P)\\
\textrm{length}(\partial^2 P_\a) &=& \textrm{length}(\partial^2 P).
\end{eqnarray*}
This uniquely determines $P_\a$ up to orientation-preserving isometry.

\begin{figure}
\begin{center} \includegraphics[width=5 in]{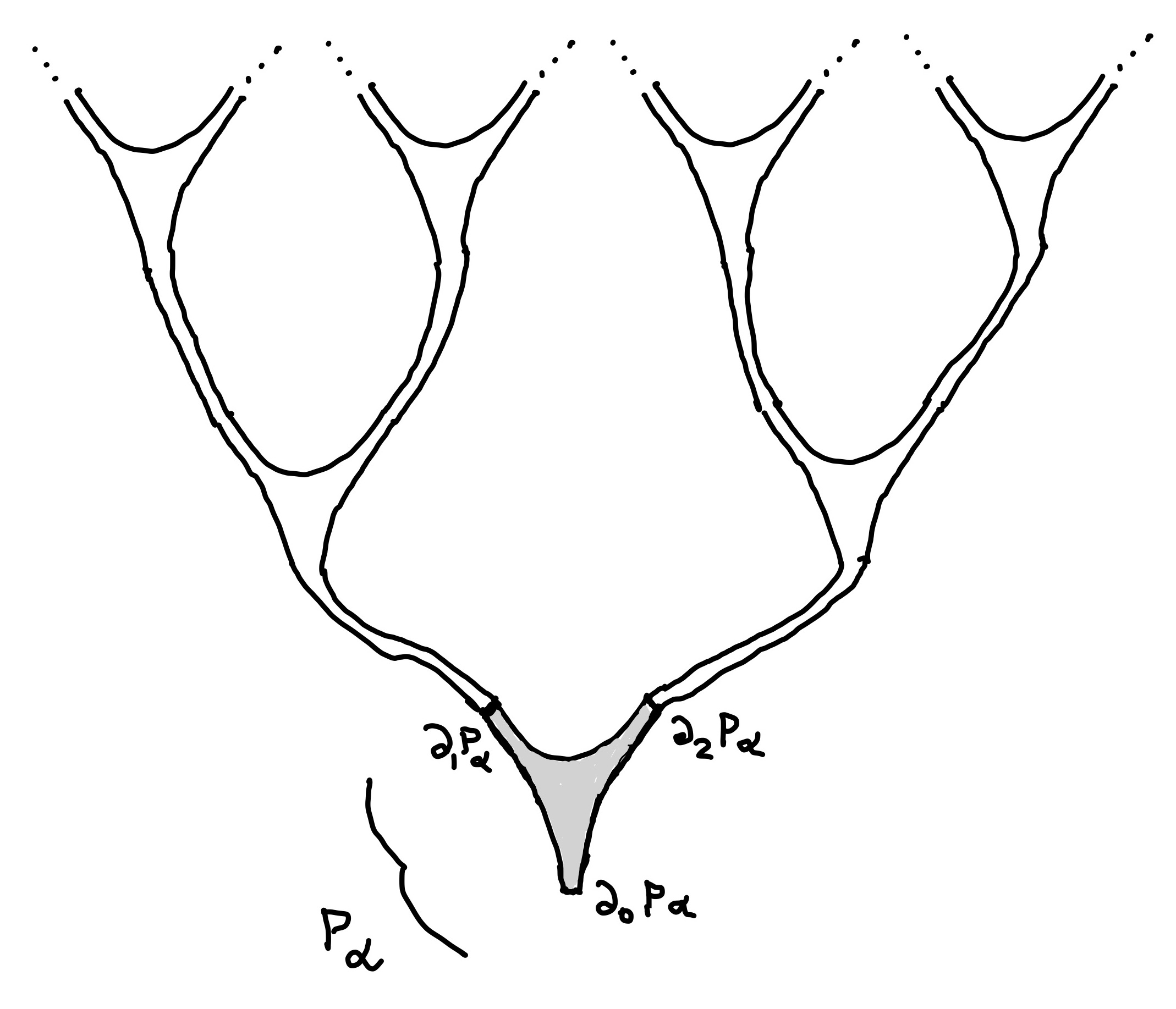} \end{center}
\caption{The surface $S_\alpha$}
\end{figure}

Define a local isometry $\psi: \partial^1 P_\a \cup \partial^2 P_\a \to \partial^1 P \cup \partial^2 P$ as follows. There exists a unique shortest geodesic $\g$ in $P$ from  $\partial^1 P$ to $\partial^2 P$. Let $p^i$ be the point of intersection of $\g$ with $\partial^i P$. Similarly, let $\g_\a$ be the unique shortest geodesic in $P_\a$ from $\partial^1 P_\a$ to $\partial^2 P_\a$. Let $p_\a^i$ be the point of intersection of $\g_\a$ with $\partial^i P_\a$. Finally, let $\psi$ be the map defined by
\begin{itemize}
\item for $i=1,2$, the restriction of $\psi$ to $\partial^i P_\a $ is an isometry onto $\partial^i P$,
\item $\psi(p_\a^i)=p^i$,
\item $\psi$ preserves orientation, where the orientation on $\partial P$ is induced from the given orientation on $P$ and the orientation on $\partial P_\a$ is induced from the given orientation on $P_\a$.
\end{itemize}
This uniquely specifies $\psi$. 

Finally, let $S_\a = (S\setminus \textrm{int}(P)) \cup P_\a/\{x \sim \psi(x)\}$ be the surface obtained from ($S$ minus the interior of $P$) and $P_\a$ by gluing together along $\psi$. 


\subsection{Continuity}\label{sec:continuity}
This subsection studies how the averages $\b_r f$ vary with $\a$ when $f$ is a function on $S_\a$. To make this precise, let $i_\a: S\setminus \mathrm{int}(P) \to S_\a$ be the inclusion map. For $f \in L^1(S\setminus \mathrm{int}(P))$, define $f_\a \in L^1(S_\a)$ by 
\begin{displaymath}
f_\a(x)= \left\{ \begin{array}{cc}
f(i_\a^{-1}(x)) &  x\in S_\a \setminus \mathrm{int}(P_\a) \\
0 & \textrm{ otherwise} \end{array}\right.
\end{displaymath}

\begin{prop}\label{prop:continuity}
Let $(S,P)$ be a panted surface and $f \in L^\infty(S\setminus \mathrm{int}(P))$. For any $r>0$, the map
$$(x,\a) \mapsto \b_r f_\a(i_\a(x))$$
is continuous as a map from $(S\setminus P) \times [0,\infty)$ to $\C$. 
\end{prop}

To begin, we introduce notation for describing the universal covers of the surfaces $S_\a$ and their deck-transformation groups. For $i=1,2$, let $v_\a^i$ be the unit tangent vector based at $p_\a^i$, tangent to $\g_\a$ and oriented so that geodesic flow moves $v_\a^i$ immediately into $\g_\a$.


Fix a unit tangent vector $w^1$ in the tangent bundle of $\H^2$. Because $S_\a$ is connected, there exists a unique orientation-preserving universal covering map $\pi_\a:X_\a \to S_\a$ such that
\begin{itemize}
\item $X_\a \subset \H^2$ is a closed simply-connected subset containing the base point of $w^1$,
\item the derivative of $\pi_\a$ maps $w^1$ to $v_\a^1$. 
\end{itemize}

Let $\tgamma_\a$ be the component of $\pi_\a^{-1}(\g_\a)$ that contains the basepoint of $w^1$. Let $w_\a^2$ be the unit vector based at the other end point of $\tgamma_\a$ so that geodesic flow moves $w_\a^2$ immediately into $\tgamma_\a$. Then the derivative of $\pi_\a$ maps $w_\a^2$ to $v_\a^2$. Let $g_{\a}$ be the unique orientation-preserving isometry of the hyperbolic plane that maps $w_0^2$ to  $w_\a^2$.

 Let $S^1_\a,S^2_\a$ be the two connected components of $S_\a \setminus \textrm{int}(P_\a)$, indexed so that $\partial^i P_\a \subset S^i_\a$ for $i=1,2$. To make the notation uniform, set $w^1_\a=w^1$. Then let $X^i_\a \subset X_\a$ be the connected component of $\pi^{-1}_\a(S^i_\a)$ that contains the base point of $w_\a^i$. So the restriction of $\pi_\a$ to $X^i_\a$ is the universal cover of $S^i_\a$. Note that $X^1_\a=X^1$ and $X^2_\a = \g_\a X^2$ for all $\a$.

Define the deck-transformation groups
\begin{eqnarray*}
\La_\a^i &=& \{g \in \Isom^+(\H^2):~\pi_\a \circ g = \pi_\a \textrm{ and } gX^i_\a = X^i_\a\} \\
\La_\a &=&  \{g \in \Isom^+(\H^2):~\pi_\a \circ g = \pi_\a \}.
\end{eqnarray*}
By Van Kampen's Theorem, $\La_\a$ is generated by $\La_\a^1$ and $\La_\a^2$. Indeed, it is the free product of these subgroups. So there is a unique isomorphism $\phi_\a:\La_0 \to \La_\a$ defined by
\begin{displaymath}
\phi_\a(g) = \left\{ \begin{array}{cc}
g & \textrm{ if } g \in \La_0^1\\
g_\a g g_\a^{-1} & \textrm{ if } g \in \La_0^2
\end{array} \right. \end{displaymath}
To simplify notation, we will drop the subscripts when they equal zero. For example, $S=S_0, \La=\La_0$, and so on.

\begin{lem}\label{lem:compact}
For every $\tx \in X^1$, radius $r>0$, $\a_{\max}\ge 0$ and $i \in \{1,2\}$ there exists a finite subset $F \subset \La$ such that the ball $B_r(\tx)$ has trivial intersection with $\phi_\a(g)X^i_\a$ for all $g\in \La$ with $g \notin F\La^i_\a$. In symbols,
$$\bigcup_{i=1}^2 \bigcup_{0\le \a \le \a_{\max}} \bigcup_{g \in \La \setminus F\La^i } B_r(\tx) \cap \phi_\a(g)X_\a^i = \emptyset.$$

\end{lem}

\begin{proof}


Let $i_0 \in \{1,2\}$, $0\le \a \le \a_{\max}$ and let $\l:[0,r'] \to \H^2$ be a unit-speed geodesic from $\tx$ to a point in $B_r(\tx) \cap \phi_\a(h)X_\a^{i_0}$ for some $h \in \La$ (and $r' \le r$). It suffices to show there is a finite set $F \subset \La$ such that $h \in F \La^{i_0}$ and $F$ does not depend on $\a$ (although it may depend on $\a_{\max}$ and $r$). 


If  the image $\pi_\a(\l) \subset S_\a$ is contained in $S^1_\a$ then $i_0=1$ and $h \in \La^1$. So in this case, we may let $F=\{1_\La\}$ and we are done.  

So we assume $\pi_\a(\l)$ is not contained in $S^1_\a$. This implies $\pi_\a(\l)$ is transverse to $\partial^1 P_\a \cup \partial^2 P_\a$. So there is a maximal discrete set  $0\le t_0<t_1<\cdots < t_n \le r'$ of times satisfying $\pi_\a(\l(t_i)) \in \partial^1 P_\a \cup \partial^2 P_\a$. Suppose $\pi_\a(\l(t_i)) \in \partial^j P_\a$ for some $j \in \{1,2\}$. Then there exist one or two elements $g \in \La$ such that
\begin{eqnarray}\label{eqn:fund-domain}
d_{\H^2}(\l(t_i), \phi_\a(g) \tp^{j}_\a) \le  \textrm{length}(\partial^j P)/2
\end{eqnarray}
where $\tp^{j}_\a$ is the basepoint of $w^j_\a$. Choose an element $g_i \in \La$ satisfying this inequality. Note $g_n \La^{i_0} = h\La^{i_0}$.  So it suffices to prove: for each $i$ with $1\le i <n$:
\begin{enumerate}
\item there exists a finite set $F \subset \La$ (depending only on $r$ and $\a_{\max}$) such that $g_i^{-1}g_{i+1} \in F$;
\item there is a $\d_0>0$ (depending only on $r$ and $\a_{\max}$) such that $t_{i+1}-t_i \ge \d_0$.
\end{enumerate}
Indeed, these claims imply $g_n \in F^n$ and $n \le r/\d_0$. 

To begin, we translate the problem to a neighborhood of $\{\tp_\a^1, \tp_\a^2\}$ as follows. To ease notation, let $\eps_i \in \{1,2\}$ be such that $\pi_\a(\l(t_i)) \in \partial^{\eps_i} P_\a$ and let 
$$\ell  = \max( \textrm{length}(\partial^1 P), \textrm{length}(\partial^2 P)).$$
By the triangle inequality,
\begin{eqnarray}\label{eqn:triangle}
&&d_{\H^2}(\tp^{\eps_i}_\a, \phi_\a(g_i^{-1}g_{i+1}) \tp^{\eps_{i+1}}_\a)\\
 &\le&  d_{\H^2}(\tp^{\eps_i}_\a, \phi_\a(g_i^{-1}) \l(t_i)) + d_{\H^2}(\phi_\a(g_i^{-1}) \l(t_i), \phi_\a(g_i^{-1}) \l(t_{i+1})) \\
 &&  + d_{\H^2}(\phi_\a(g_i^{-1}) \l(t_{i+1}), \phi_\a(g_i^{-1}g_{i+1}) \tp^{\eps_{i+1}}_\a) \\
 &\le& \ell + r
\end{eqnarray}
where the last inequality comes from two applications of (\ref{eqn:fund-domain}) and the fact that $d_{\H^2}(\l(t_{i+1}), \l(t_i)) \le r$. 

\noindent {\bf Case 1}. Suppose the geodesic segment $\pi_\a(\l[t_i, t_{i+1}])$ is contained in $S^j_\a$ for some $j \in \{1,2\}$. 

In this case, there is a positive lower bound on the length $t_{i+1}-t_i$ because the surface  $S^j_\a$ does not depend on $\a$ (up to isometry) and $t_{i+1}-t_i$ is at least as large as the shortest curve in $S^j$ from $\partial^j P$ to itself that is not homotopic into the boundary. 


If $\pi_\a(\l[t_i, t_{i+1}])$ is contained in $S^1_\a=S^1$ then (\ref{eqn:triangle}) reduces to
$$d_{\H^2}(\tp^1, g_i^{-1}g_{i+1} \tp^1) \le \ell +r.$$
This is because $g_i^{-1}g_{i+1}\in \La^1$, $\phi_\a$ is the identity on $\La^1$ and $\tp^1_\a=\tp^1$. Since $\La^1$ is discrete, there are only finitely many elements of $\La^1$ that move $\tp^1$ by distance at most $\ell+r$. 

If $\pi_\a(\l[t_i, t_{i+1}])$ is contained in $S^2_\a$ then (\ref{eqn:triangle}) reduces to
$$d_{\H^2}(\tp^2, g_i^{-1}g_{i+1} \tp^2) \le \ell +r.$$
This is because $g_i^{-1}g_{i+1}\in \La^2$, $\phi_\a(g_i^{-1}g_{i+1}) = g_\a g_i^{-1}g_{i+1} g_\a^{-1}$ and $\tp^2_\a = g_\a \tp^2$ (and the hyperbolic metric is left-invariant so we can cancel the $g_\a$'s). Since $\La^2$ is discrete, there are only finitely many elements of $\La^2$ that move $\tp^2$ by distance at most $\ell+r$. This finishes Case 1.

\noindent {\bf Case 2}. Suppose the geodesic segment $\pi_\a(\l[t_i, t_{i+1}])$ is contained in $P_\a$.

Suppose $\pi_\a(\l[t_i, t_{i+1}])  = \gamma_\a$. Then $g_i=g_{i+1}$, so we can choose $F$ to consist of the identity element.  By equation (\ref{hexagon}) applied to either of the canonical right-angled hexagons inside $P_\a$,
\begin{eqnarray}\label{hexagon2}
\cosh(\a)&=&\sinh( \textrm{length}(\partial^1 P)/2) \sinh( \textrm{length}(\partial^2 P)/2) \cosh( \textrm{length}(\gamma_\a))\\
&& - \cosh( \textrm{length}(\partial^1 P)/2) \cosh( \textrm{length}(\partial^2 P)/2).
\end{eqnarray}
Since $\cosh(\a) \ge 1$, 
$$\cosh(\textrm{length}(\gamma_\a)) \ge \frac{1+  \cosh( \textrm{length}(\partial^1 P)/2) \cosh( \textrm{length}(\partial^2 P)}{\sinh( \textrm{length}(\partial^1 P)/2) \sinh( \textrm{length}(\partial^2 P)/2)} > 1.$$
So the length of $\gamma_\a$ admits a positive lower bound that does not depend on $\a$. Since $\pi_\a(\l[t_i, t_{i+1}])  = \gamma_\a$ this implies a positive lower bound on $t_{i+1}-t_i$ that does not depend on $\a$. 

So assume $\pi_\a(\l[t_i, t_{i+1}])  \ne \gamma_\a$. Let $e_{jk}$ be the shortest geodesic segment from $\partial^j P_\a$ to $\partial^k P_\a$ (for $j,k \in \{0,1,2\}$). This is well-defined even when $\a=0$ by the requirement that $e_{0j}$ meets $\partial^j P_\a$ in a right-angle for $j \in \{1,2\}$. Note $e_{12}=\g_\a$. 

Since $\pi_\a(\l[t_i, t_{i+1}])  \ne \gamma_\a$, $\pi_\a(\l[t_i, t_{i+1}])$ is transverse to  $\cup_{j,k} e_{jk}$. So there exists a maximal set of times $t_i < s_1<s_2<\ldots < s_m < t_{i+1}$ and elements $\eta_j \in \{ 01, 02, 12\}$ such that $\pi_\a(\l(s_j)) \in e_{\eta_j}$ for all $j$. Moreover, $g_i^{-1}g_{i+1}$ is determined by the sequence $\eta_1,\ldots, \eta_m$ of sides and $\eps_i,\eps_{i+1}$. So it suffices to show there are only finitely many such sequences possible. To do this, it suffices to show there is a lower bound on $s_{j+1}-s_j$ that depends only on $\a_{\max}$ and $r$ (for all $1\le j <m$). This also implies the required lower bound on $t_{i+1}-t_i$. 

Suppose $12 \in \{\eta_j, \eta_{j+1}\}$.  In this case, $\pi_\a(\l[s_j,s_{j+1}])$ is a geodesic from a point in $e_{12}=\g_\a$ to a segment of the form $e_{0k}$ for some $k \in \{1,2\}$. But the shortest geodesic from $\g_\a$ to $e_{0k}$ is along $\partial^k P_\a$ and has length equal to half the length of $\partial^k P_\a$. Since this length does not depend on $\a$, it provides a positive lower bound on $s_{j+1}-s_j$ independent of $\a$. 

We may now assume $\{\eta_j, \eta_{j+1}\} = \{01,02\}$. Let $u_k$ be the point of intersection of $\partial^k P_\a$ with $e_{0k}$ (for $k\in \{1,2\}$). Note that $\pi_\a(\l(s_j))$ and $\pi_\a(\l(s_{j+1}))$ each have distance at most $r$ from $\{u_1,u_2\}$. 

Suppose the claim is false. By considering the canonical right-angled hexagons associated with $P_\a$, we see that for every $\eps>0$ there exist a right-angled hexagon $H_\eps$ bounded by sides $f_k, e_{kl}$ ($k,l \in \{0,1,2\}$) and points $u'_k \in e_{0k}$ satisfying
\begin{enumerate}
\item $\textrm{length}(f_k) = \textrm{length}(\partial^k P)/2$ for $k \in \{1,2\}$,
\item $\textrm{length}(f_0) \in [0,\a_{\max}]$,
\item if $u_k$ is the vertex at the intersection of $f_k$ and $e_{0k}$ then $d_{\H^2}(u_k,u'_k) \le r$,
\item $d_{\H^2}(u'_1,u'_2) \le \eps$. 
\end{enumerate}
Here, the points $u'_1, u'_2$ correspond with $\pi_\a(\l(s_j))$ and $\pi_\a(\l(s_{j+1}))$. See figure \ref{fig:pants}.

\begin{figure}\label{fig:pants}
\begin{center} \includegraphics[width=3 in]{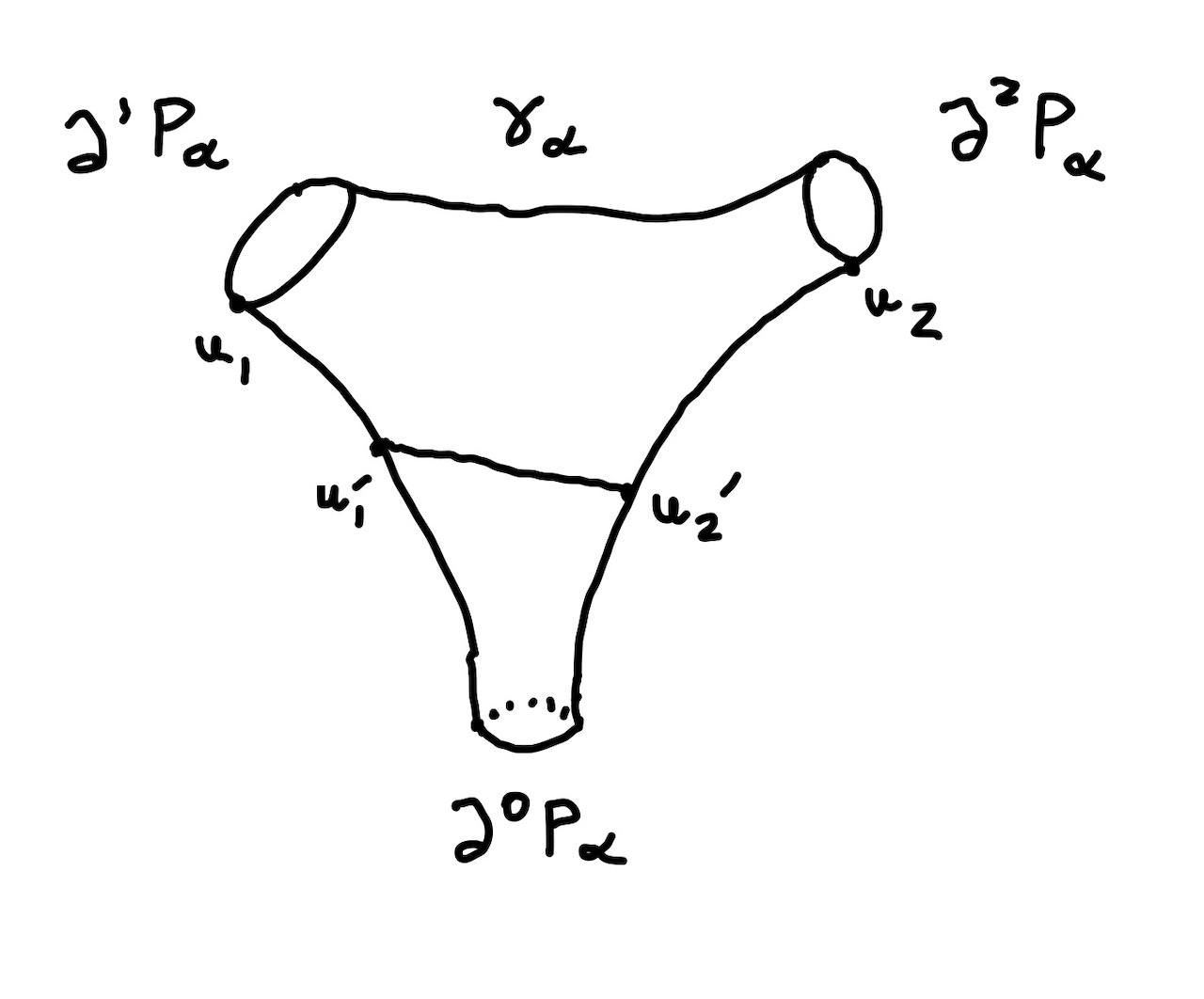} \end{center}
\caption{The pair of pants $P_\a$.}
\end{figure}

By (\ref{hexagon}), the length of $e_{12}$ is bounded from above and below by positive constants depending only on $\a_{\max}$. Thus the sides $f_1,e_{12},f_2$ and points $u'_1,u'_2$ are all contained in a ball $B$ whose radius is bounded in terms of $\a_{\max}, r$ and the constants $\textrm{length}(\partial^k P)$ ($k \in \{1,2\}$). Let us consider $u_1$ to be fixed in the hyperbolic plane (independent of $\eps$) and consider taking a subsequential limit of these hexagons as $\eps \searrow 0$ in the Fell topology. The limit polygon is such that its sides $e_{01}$ and $e_{02}$ intersect in $\H^2$. So it is a compact convex pentagon. However, it is not possible to obtain a compact pentagon as a limit of right-angled hexagons (even allowing that some of the sides of the right-angled hexagons have zero length). Indeed, if it was possible then it would be possible to do it with right-angled hexagons of bounded diameter such that at least one of the side-lengths tends to zero in the limit. But the formula (\ref{hexagon}) shows that for every $D>0$ there is $\d>0$ such that if a right-angled hexagon $H$ has a side, say $e_{12}$, with length $<\d$ then the diameter of $H$ is $>D$. This contradiction shows that there is a positive lower bound on $s_{j+1}-s_j$ depending only on $r$ and $\a_{\max}$ as required. This finishes the last case and therefore, finishes the proof. 

\end{proof}

\begin{proof}[Proof of Proposition \ref{prop:continuity}]

The quantity $\b_r f_\a(i_\a(x))$ is uniformly continuous in $x$. Indeed, suppose for some $j \in \{1,2\}$, $x,y \in S^j$. Let $\pi_\a: X_\a \to S_\a$ be the universal covering map and let $\tx, \ty \in X_\a \subset \H^2$ be lifts of $i_\a(x),i_\a(y)$ such that $d_{S_\a}(i_\a(x),i_\a(y))=d_{\H^2}(\tx,\ty)$. Then 
\begin{eqnarray*}
|\b_r(f_\a)(i_\a(x)) - \b_r(f_\a)(i_\a(y))|  &=& \frac{1}{\area(B_r(\tx))} \left| \int_{B_r(\tx)} \tf(z)~dz - \int_{B_r(\ty)} \tf(z)~dz \right| \\ 
&\le& \|f\|_\infty \frac{\area(B_r(\tx) \vartriangle B_r(\ty))}{\area(B_r(\tx))}
\end{eqnarray*}
where $\vartriangle$ denotes symmetric difference. Because the map $i_\a$ restricted to $S^j$ is an isometry the distance $d_{S_\a}(i_\a(x),i_\a(y)) = d_S(x,y)$. Since the bound above tends to zero uniformly in the distance $d_{S_\a}(i_\a(x),i_\a(y))$, this proves the claim. Therefore, it suffices to prove that for any fixed $x \in S \setminus P$, the map $\a \mapsto \b_r f_\a(i_\a(x))$ is continuous.



Recall
$$\b_r f_\a(i_\a(x)) = \textrm{area}(B_r(\tx))^{-1} \int_{B_r(\tx)} \tf_\a(y) \deee y$$
where $\tx$ is a preimage of $i_\a(x)$. By symmetry, we may assume that $x \in S^{1}_\a$. Since $X^1_\a = X^1$ for all $\a$, we can choose $\tx \in X^1$ so that it does not depend on $\a$. 

Note that the preimage of $S^{1}_\a \cup S^{2}_\a$ in $X_\a$ is the disjoint union of the translates of $X^1_\a$ and $X^2_\a$. In symbols,
$$\bigcup_{i=1}^2 \bigcup_{g \La^i \in \La / \La^i} \phi_\a(g)X_\a^i. $$
So 
$$\int_{B_r(\tx)} \tf_\a(y) \deee y = \sum_{i=1}^2 \sum_{g \La^i \in \La / \La^i} \int_{B_r(\tx) \cap \phi_\a(g)X_\a^i} \tf_\a(y) \deee y.$$
By Lemma \ref{lem:compact} there are finite sets $F^1,F^2 \subset \La$ (depending only on an upper bound for $\a$ and $r$) such that
\begin{eqnarray}\label{finite}
\int_{B_r(\tx)} \tf_\a(y) \deee y = \sum_{i=1}^2 \sum_{g \in F^i} \int_{B_r(\tx) \cap \phi_\a(g)X_\a^i} \tf_\a(y) \deee y.
\end{eqnarray}
The integrals can be rewritten as follows:
\begin{eqnarray}\label{rewritten}
\int_{B_r(\tx) \cap \phi_\a(g)X_\a^i} \tf_\a(y) \deee y = \int_{\phi_\a(g^{-1})B_r(\tx) \cap X_\a^i} \tf_\a(\phi_\a(g)y) \deee y = \int_{\phi_\a(g^{-1})B_r(\tx) \cap X_\a^i} \tf_\a(y) \deee y
\end{eqnarray}
where the first equality follows from the change of variables $y \mapsto \phi_\a(g)y$ and the second from the $\La_\a$-invariance of $\tf_\a$. If $i=1$ then $\tf_\a(y)=\tf(y)$ for all $y \in X^1_\a=X^1$. So
$$\int_{B_r(\tx) \cap \phi_\a(g)X_\a^1} \tf_\a(y) \deee y = \int_{\phi_\a(g^{-1})B_r(\tx) \cap X^1} \tf(y) \deee y.$$
If $i=2$ then $\tf_\a(g_\a y) = \tf(y)$ for $y \in X^2$ (and $g_\a X^2 = X_\a^2$). By a change of variables
\begin{eqnarray*}
\int_{\phi_\a(g^{-1})B_r(\tx) \cap X_\a^2} \tf_\a(y) \deee y &=& \int_{\phi_\a(g^{-1})B_r(\tx) \cap X_\a^2} \tf(g_\a^{-1}y) = \int_{g_\a^{-1} \phi_\a(g^{-1})B_r(\tx) \cap X^2} \tf(y) \deee y.
\end{eqnarray*}

Combined with  (\ref{finite}) and (\ref{rewritten}) this implies
\begin{eqnarray*}
 \textrm{area}(B_r(\tx))  \b_r f_\a(i_\a(x)) &=& \sum_{g \in F^1 }\int_{\phi_\a(g^{-1})B_r(\tx) \cap X^1} \tf(y) \deee y, + \sum_{g \in F^2 }\int_{g_\a^{-1} \phi_\a(g^{-1})B_r(\tx) \cap X^2} \tf(y) \deee y.
 \end{eqnarray*}
 Observe that each of the integrals above is continuous in $\a$ because $\a \mapsto g_\a$ and $\a \mapsto \phi_\a(g)$ are continuous (for fixed $g$). So we have expressed $\b_rf_\a(i_\a(x))$ as a finite sum of functions that are continuous in $\a$. Thus $\b_r \tf_\a(i_\a(x))$ is continuous in $\a$.

\end{proof}




\section{Averaging around cusps}\label{sec:cusps}

The main result of this section is a comparison between the averages of the form $\b_r(f)$ and $\b_r(f1_C)$ where $C$ is a cusp of the surface. This is used in the proof of Lemma \ref{geometric} to control the maximal function under these kinds of deformations of functions. To be precise, we need the following definitions. 

Let $C = H_0/\{z \mapsto z+x_0\}$ be a cusp where $H_0 =\{x+iy \in \H^2:~y \ge 1\}$ is the canonical horoball and $x_0>0$ is the length of the boundary of $C$ (which is a horocycle). For $t>0$, let 
$$C[t]=\{x+iy \in \H^2:~ y \ge e^t\}/\{z \mapsto z+x_0\} \subset C.$$
This is the unique cusp contained in $C$ such that the distance between the boundaries $\partial C$ and $\partial C[t]$ is $t$.

\begin{prop}\label{prop:flow}
Let $S$ be a hyperbolic surface with pairwise disjoint cusps $C_1,\ldots, C_k \subset S$. Let $U=\cup_{i=1}^k C_i$ be the union of the cusps and $U[t] = \cup_{i=1}^k  C_i[t]$ the union of the shortened cusps for $t \ge 0$. Let $f \in L^\infty(S)$ be a non-negative function such that (1) $f$ is constant on $C_i$ for all $i$ and (2) $f(p)=0$ for all $p \in S \setminus U$.  Then for all $p \in S \setminus U$ and $t, r \ge 0$,
$$\b_{r+t}(f1_{U[t]})(p) \ge e^{-t}(1-2e^{-r}) \b_r(f)(p).$$
\end{prop}

\begin{proof}
 
Because $\b_r$ is linear, it suffices to consider the special case in which $f(p)=1$ for all $p \in U$. By passing to the universal cover, it suffices to prove: for any $p \in \H^2 \setminus H_0$, 
$$\frac{\area(B(r+t,p) \cap \{x+iy:~y \ge e^t\} ) }{\area(B(r+t,p))} \ge e^{-t}(1-2e^{-r}) \frac{\area(B(r,p) \cap H_0 ) }{\area(B(r,p))}.$$

Before estimating the above, here are some general facts about areas of intersections of balls and horoballs. 

For $R>T> 0$, let $g(R,T)$ be the area of the intersection of a ball $B$ and a horoball $H$ such that the radius of $B$ is $R$ and the distance between the center of $B$ and the boundary of $H$ is $T$. Then $g(R,T)$ is well-defined (in that it depends on the choice of $B$ and $H$ only through $R$ and $T$) and for any fixed $t_0$, $g(T+t_0,T)$ is monotone increasing in $T$. To see this, we may assume $H=H_0$ and $t_0>0$ (since if $t_0 \le 0$ then $g(T+t_0,T)=0$). Set $B_T$ equal to the ball of hyperbolic radius $T+t_0$ and hyperbolic center  $e^{-T}i$ in the upper half-plane model $\H^2$. Recall that the hyperbolic distance between two points on the imaginary axis is the absolute difference between their logarithms (so $d_{\H^2}( e^a i, e^b i ) = |a-b|$). So $g(T+t_0,T)=\area(H_0 \cap B_T)$. Also $B_T$ coincides with the Euclidean disk centered on the imaginary axis that contains $e^{t_0}i$ and $e^{-2T-t_0}i$ in its boundary. In particular, $B_T \subset B_{T'}$ for any $T \le T'$. So $g(T+t_0,T) \le g(T'+t_0,T')$.

It follows that 
$$\area\left(B(r+t,p) \cap \{x+iy:~y \ge e^t\} \right) = g(r+t, d_{\H^2}(p,H_0)+t) \ge g(r,d_{\H^2}(p,H_0)) = \area(B(r,p) \cap H_0).$$
So it suffices to show
$$\frac{\area(B(r,p))}{\area(B(r+t,p))} \ge e^{-t}(1-2e^{-r}).$$
Since $\area(B(r,p)) = 2\pi(\cosh(r)-1)$, 
\begin{eqnarray*}
\frac{\area(B(r,p))}{\area(B(r+t,p))}  &=& \frac{\cosh(r)-1}{\cosh(r+t)-1} = \frac{e^r - 2 + e^{-r}}{e^{t+r}-2+e^{-t-r}} \\
&\ge& \frac{e^r-2}{e^{t+r}} =  e^{-t}(1-2e^{-r}).
\end{eqnarray*}

\end{proof}

\section{The inductive step}

To prove Lemma \ref{geometric}, we will construct surfaces $S$ with functions $f \in L^1(S)$ by induction. To be precise, we need the next two definitions.

\begin{defn}
A tuple $\left(S,P,\{C_i\}_{i=1}^k, U, f\right)$ is {\bf good} if
\begin{enumerate}
\item $(S,P)$ is a panted surface,
\item $S$ is a complete hyperbolic surface with finite area and no boundary,
\item $C_1,\ldots, C_k \subset S$ are pairwise disjoint cusps,
\item $P$ is disjoint from $U=\cup_i C_i$,
\item $f \in L^1(S)$ is non-negative,
\item $f$ is constant on each cusp $C_i$,
\item $f(p) = 0$ for all $p \in S \setminus U$,
\item $\|f\|_1 \le 2$. 
\end{enumerate}
\end{defn}

\begin{defn}
For $\rho \ge 0$ and $f \in L^1(S)$, let
$$\mathsf{M}_\rho f(p) = \sup_{\rho \le r} \b_r(|f|)(p)$$
be the {\bf $\rho$-truncated maximal function} of $f$.  
\end{defn}

The next result forms the inductive step in the proof of Lemma \ref{geometric}.
\begin{prop}\label{inductive}
Let $\left(S,P, \{C_i\}_{i=1}^k, U, f\right)$ be a good tuple and let $\rho,\eps$ be parameters such that $10\le \rho$ and $0<\eps<1/10$. Let 
$$V=\left\{p \in S \setminus (P \cup U):~\mathsf{M}_\rho f(p) \ge 1\right\}.$$
Then there exists a good tuple $\left(\hS,\hP,\{\hC_j\}_{j=1}^{2k}, \hU, \hf\right)$ satisfying
\begin{enumerate}
\item $\area(\hS)  = 2\, \area(S) + 2\pi$,
\item if 
$$\hV=\left\{p \in \hS \setminus (\hP \cup \hU):~\mathsf{M}_\rho \hf(p) \ge 1\right\}$$
then $\area(\hV) \ge 2\, \area(V)-3\eps$,
\item $\displaystyle \| \hf\|_1 \le \frac{\|f\|_1(1- \|f\|_1/6)}{1-4\eps-4e^{-\rho}}$. 
\end{enumerate}
\end{prop}

\begin{proof}
By definition of $V$, there exist $R>0$ and a compact subset $W \subset V$ such that $\area(W) \ge \area(V) - \eps$ and
$$\sup_{\rho \le r \le R} \b_r(f)(p) \ge 1-\eps$$
for all $p \in W$. 

By Proposition \ref{prop:continuity}, there exists $\a>0$ such that if $S_\a$ and $f_\a$ are defined as in \S \ref{sec:continuity} then
$$\sup_{\rho \le r \le R} \b_r(f_\a)(p) \ge 1-2\eps$$
for all $p \in W$. Here we are identifying $W$ with a subset of $S_\a$. This makes sense because $S \setminus P$ is naturally isometric to $S_\a\setminus P_\a$ and $W \subset V \subset S \setminus P$. 

Let $S^{(1)}, S^{(2)}$ be two isometric copies of $S_\a$. For $i=1,2$ and $1\le j \le k$, let $C^{(i)}_j\subset S^{(i)}$ be the copy of the cusp $C_j$ in $S^{(i)}$ and let $f^{(i)} \in L^1(S^{(i)})$ be a copy of $f_\a$. Define $V^{(i)}, U^{(i)}, W^{(i)} \subset S^{(i)}$ similarly.

The surface $S_\a$ has a single boundary component which is of length $\a$. Let $Y_\a$ be the pair of pants with one cusp and two geodesic boundary components $\partial^1 Y_\a$ and $\partial^2 Y_\a$, both of length $\a$. For $i=1,2$, let $\psi^{(i)}:\partial^i Y_\a \to \partial S^{(1)}$ be an isometry and let $\psi:\partial Y_\a \to \partial(S^{(1)} \sqcup S^{(2)})$ be the union of these two maps. Finally, let 
$$\hS = \left(S^{(1)} \sqcup S^{(2)} \sqcup Y_\a\right)/\{x \sim \psi(x)\}$$
be the result of gluing $Y_\a$ to $S^{(1)} \sqcup S^{(2)} $ via $\psi$. Let $\hP$ be the copy of $Y_\a$ in $\hS$. Conclusion (1) is immediate.

Extend $f^{(i)}$ to all of $\hS$ by setting $f^{(i)}(p) = 0$ for all $p \in \hS \setminus S^{(i)}$. By Nevo's Pointwise Ergodic Theorem (Theorem \ref{thm.2})  applied to $f^{(1)}$, there exists $t>0$ and $W' \subset W^{(2)}$ such that $\area(W') \ge \area(W^{(2)}) - \eps$ and for all $p \in W'$ and $r \ge t$,
$$\b_r\left(f^{(1)}\right)(p) \ge -\eps + \int f^{(1)}~d\nu_{\hS}.$$
Define cusps
$$\hC_j : = C^{(1)}_j, \quad \hC_{k+j} := C^{(2)}_j[t]$$
for $1\le j \le k$. 

Define $\barf \in L^1(\hS)$ by
$$\barf =    f^{(1)} + \left[1 - \int f^{(1)}~d\nu_{\hS}\right]e^t 1_{U^{(2)}[t]}  f^{(2)}$$
where $U^{(2)}[t] = \cup_{j=1}^k C^{(2)}_j[t]$ is as defined in \S \ref{sec:cusps}.

Because $\|f\|_1 \le 2$ (by definition of a good tuple), it follows that 
$$1 - \int f^{(1)}~d\nu_{\hS} = 1 - \frac{\area(S)}{\area(\hS)} \int f~d\nu_{S} > 0.$$
So both summands defining $\barf$ are non-negative. In particular, $\barf\ge 0$. 

Set
$$\hf := \frac{ \barf}{ 1 - 4\eps  - 4 e^{-\rho}}.$$
It is immediate that $\left(\hS,\hP,\{\hC_j\}_{j=1}^{2k}, \hU, \hf\right)$ is a good tuple.

The next step is to verify the maximal function estimates. We claim that if $p \in W^{(1)} \cup W'$ then $\mathsf{M}_\rho\hf(p) \ge 1$. So suppose  $p \in W^{(1)}$. Then the definition of $W$ implies
$$\mathsf{M}_\rho\barf(p) \ge \mathsf{M}_\rho f^{(1)}(p) \ge 1-2\eps.$$
Therefore
\begin{eqnarray}\label{eqn:m1}
\mathsf{M}_\rho\hf(p)  \ge \frac{1-2\eps}{1 - 4\eps  - 4 e^{-\rho}} \ge 1.
\end{eqnarray}

If $p \in W' \subset W^{(2)}$, then there exists $r \ge \rho$ such that
$$\b_r\left(f^{(2)}\right)(p) \ge 1-\eps.$$
By Proposition \ref{prop:flow},
$$\b_{r+t}\left(1_{U^{(2)}[t]}f^{(2)}\right)(p) \ge e^{-t}(1-2e^{-r}) \b_r\left(f^{(2)}\right)(p) \ge e^{-t}(1-2e^{-r})(1-\eps).$$
Therefore,
\begin{eqnarray*}
\mathsf{M}_\rho \barf(p) &\ge& \b_{r+t}(\barf)(p) \ge \b_{r+t}\left(f^{(1)}\right)(p) + \left[1 - \int f^{(1)}~d\nu_{\hS}\right] e^t \b_{r+t}\left(1_{U^{(2)}[t]}f^{(2)}\right)(p) \\
&\ge& -\eps + \int f^{(1)}~d\nu_{\hS} + \left[1 - \int f^{(1)}~d\nu_{\hS}\right] (1-2e^{-r})(1-\eps) \\
&= & -\eps + (1-2e^{-r})(1-\eps) + \left(\int f^{(1)}~d\nu_{\hS}\right) \left[ 1- (1-2e^{-r})(1-\eps)\right] \\
&\ge& 1 - 3\eps  - 4 e^{-r} \ge 1 - 4\eps  - 4 e^{-\rho}\label{eqn:m2}
\end{eqnarray*}
where the lower bound on $\b_{r+t}\left(f^{(1)}\right)(p)$ follows from the definition of $W'$.  Therefore, $\mathsf{M}_\rho \hf(p) \ge 1$. Together with inequality (\ref{eqn:m1}) this implies $\mathsf{M}_\rho \hf(p) \ge 1$ for all $p \in W^{(1)} \cup W'$. So $\hV \supset W^{(1)} \cup W'$ which implies
$$\area(\hV) \ge 2\, \area(V)-3\eps.$$
This verifies conclusion (2).
 
Next, we verify conclusion (3). Recall that our normalization conventions imply $\area(\hS) \|f^{(1)}\|_1 = \area(S) \|f\|_1$ (for example). Because $\area(C[t]) = e^{-t}\area(C)$ for any cusp $C$,
$$\area(\hS)\left\|1_{U^{(2)}[t]}  f^{(2)}\right\|_1 = \area(S) e^{-t} \|f\|_1.$$
So
 \begin{eqnarray*}
\area(\hS) \| \barf\|_1 &=& \area(\hS) \|f^{(1)}\|_1 + \area(\hS) \left[1 - \int f^{(1)}~d\nu_{\hS}\right]e^t \left\|1_{U^{(2)}[t]}  f^{(2)}\right\|_1\\
&=& \area(S) \|f\|_1 + \area(S) \left[1 - \int f^{(1)}~d\nu_{\hS}\right] \left\|f \right\|_1\\
&=&\area(S) \|f\|_1 \left(2-  \frac{\area(S)}{\area(\hS)}\|f\|_1\right) \le \area(S) \|f\|_1 \left(2-  \|f\|_1/3\right)
\end{eqnarray*}
where the last inequality comes from the fact that $\area(\hS) = 2\area(S) + 2\pi$ and since $\hS$ contains a pair of pants, $\area(\hS)\ge 2\pi$. Therefore,  $\frac{\area(S)}{\area(\hS)} \ge 1/3$. 

Divide both sides by $\area(\hS)$ and use the estimate $\area(S)/\area(\hS) \le 1/2$ to obtain
$$ \| \barf\|_1   \le \|f\|_1 (1- \|f\|_1/6) $$
which implies conclusion (3). 

\end{proof}

\section{The end of the proof}

The next lemma establishes the base case of the induction in the proof of Lemma \ref{geometric}.

\begin{lem}\label{lem:base}
For every $\rho \ge 0$, there exists a good tuple $\left(S,P, \{C_i\}_{i=1}^4, U, f\right)$ such that
$$\nu_S\left(\{p \in S \setminus (P \cup U):~\mathsf{M}_\rho f(p) \ge 1\}\right) \ge 1/2.$$

\end{lem}

\begin{proof}
Let $\a>0$ and  let $Y_1$ be a pair of pants with two cusps and one geodesic boundary component of length $\a>0$. Let $Y_2$ be an isometric copy of $Y_1$. Let $P$ be a pair of pants with one cusp and two geodesic boundary components each of  length $\a$. Let $\psi:\partial P \to \partial Y_1 \sqcup \partial Y_2$ be an isometry and let
$$S = [Y_1 \sqcup Y_2 \sqcup P]/\{x \sim \psi(x)\}$$
be the surface obtained by gluing $Y_1,Y_2$ and $P$ together by way of $\psi$. Then $(S,P)$ is a panted surface with area $6\pi$. 

For $i=1,2,$ let $V_i \subset Y_i$ be a compact subsurface with 
$$\area(V_i) \ge 3\,\area(Y_i)/4 = 3\pi/2.$$
Let $C_1^{(i)}, C_2^{(i)} \subset Y_i$ be disjoint cusps such that for any $p \in V_i$ and $q \in C_1^{(i)} \cup C_2^{(i)}$, $d_S(p,q) \ge \rho$. Let $f \in L^1(S)$ be any non-negative function such that $\left(S,P, \{C_i\}_{i=1}^4, U, f\right)$ is a good tuple and $\|f\|_1 = 1$. For example, one could define $f$ by
\[
f(p) = \begin{dcases}
 \frac{\area(S)}{4\, \area\left(C_j^{(i)}\right)} &  p \in C_j^{(i)} \\
0 & \textrm{ otherwise} \end{dcases} \]

By Nevo's Pointwise Ergodic Theorem \ref{thm.2}, for a.e. $p \in S$, $\mathsf{M}f(p) \ge 1$. Since $\b_rf(p)=0$ for all $r < \rho$ and $p \in V_1 \cup V_2$, it follows that $\mathsf{M}_\rho f(p) \ge 1$ for all $V_1 \cup V_2$. Since 
$$\area(V_1 \cup V_2)\ge 3\pi = \area(S)/2$$
this finishes the proof.

\end{proof}

\begin{lem}\label{contraction}
Let $t_1,t_2,\ldots$ be a sequence of real numbers $t_i \in [0,2)$ such that $t_{i+1} \le t_i(1-t_i/6)$ for all $i$. Then $\lim_{i\to\infty} t_i=0$. 
\end{lem}

\begin{proof}
Since $1-t_i/6<1$, the sequence is monotone decreasing. So the limit exists $L=\lim_{i\to\infty} t_i$ exists, $L \in [0,2)$ and $L = L(1-L/6)$. This implies $L=0$.
\end{proof}

\begin{proof}[Proof of Lemma  \ref{geometric}]
For $b,\rho>0$, let $\Si(b,\rho)$ be the set of all numbers $\d>0$ such that there exists a good tuple $\left(S,P, \{C_i\}_{i=1}^k, U, f\right)$ satisfying
\begin{enumerate}
\item $f\ge 0$,
\item $\|f\|_1 \leq \d$,
\item $\nu_S\left(\left\{p \in S \setminus (P \cup U): \mathsf{M}_\rho f(p)  \ge 1\right\}\right) \geq b.$
\end{enumerate}
Also let $\overline{\Si(b,\rho)}$ denote the closure of $\Si(b,\rho)$ in $[0,\infty)$. It suffices to prove that $0 \in \overline{\Si(b, 10)}$ for some $b>0$. 

Note that if $b' \le b$ and $\rho' \ge \rho$ then $\Si(b,\rho) \subset \Si(b',\rho')$. Lemma \ref{lem:base} proves that $1 \in \Si(1/2,\rho)$ for all $\rho$. Proposition \ref{inductive} proves: if $\delta \in \Si(b,\rho)$ for all $\rho \ge 10$ then $\d(1-\d/6) \in \overline{\Si(b-\eps,\rho)}$ for all $\eps>0$ and $\rho \ge 10$. By iterating and using Lemma \ref{contraction}, this implies $0 \in \overline{\Si(1/2-\eps,\rho)}$ for all $\eps>0$ and $\rho \ge 10$ which finishes the lemma.

\end{proof}

\section{Two open problems}

The main counterexample does not have spectral gap. This is because we are forced to make the ``necks'' in the construction of the surface arbitrarily narrow. Similarly, Tao's construction does not have spectral gap. This raises a question: does Nevo's Pointwise Ergodic Theorem \ref{thm.2} hold in $L^1$ if $G \cc (X,\mu)$ has spectral gap? It also raises the converse question: if $G \cc (X,\mu)$ is ergodic but does not have spectral gap then does the Pointwise Ergodic Theorem necessarily fail in $L^1$ for this action?

\bibliography{biblio}
\bibliographystyle{alpha}

\end{document}